\documentclass{amsart}
 \usepackage{amsfonts,mathrsfs,dsfont,microtype}
 \usepackage{amsthm}
 \usepackage{amssymb}
 \usepackage{graphicx}
 \usepackage{enumerate}
 \usepackage{tikz-cd}
 \usepackage{hyperref}
  \usepackage{verbatim}

\usepackage{mathtools}
\DeclarePairedDelimiter{\floor}{\lfloor}{\rfloor}

 \hypersetup{
    pdftoolbar=true,
    pdfmenubar=true,
    pdffitwindow=false,
    pdfstartview={FitH},
    pdftitle={},
    pdfauthor={},
    pdfsubject={},
    pdfkeywords={}
    pdfnewwindow=true,
    colorlinks=true, %set this false for printable version
    linkcolor=blue,
    citecolor=blue,
    urlcolor=black,
}

\theoremstyle{definition}
\newtheorem{Def}{Definition}[section]
\newtheorem{ex}[Def]{Example}
\newtheorem{cons}[Def]{Construction}
\newtheorem{symmred}[Def]{Symmetry Reduction}
\newtheorem{rem}[Def]{Remark}

\theoremstyle{plain}
\newtheorem{prop}[Def]{Proposition}

\newtheorem*{thm*}{Theorem}
\newtheorem{mainthm}{\sc Theorem}

\newtheorem{lem}[Def]{Lemma}
\newtheorem{cor}[Def]{Corollary}
\newtheorem*{cor*}{Corollary}
\newtheorem{con}[Def]{Conjecture}
\newtheorem*{con*}{Conjecture}
\newtheorem*{frag*}{Question}

\renewcommand{\O}{\mathsf O}
\newcommand{\SO}{\mathsf{SO}}
\newcommand{\SU}{\mathsf{SU}}
\newcommand{\G}{\mathsf{G}}
\newcommand{\K}{\mathsf{K}}
\newcommand{\GL}{\operatorname{{\mathsf GL}}}

\newcommand{\C}{{\mathds C}}
\newcommand{\R}{{\mathds R}}
\newcommand{\N}{{\mathds N}}
\newcommand{\cC}{{\mathcal C}}
\newcommand{\cO}{{\mathcal O}}
\newcommand{\cV}{{\mathcal V}}
\newcommand{\Sym}{\operatorname{Sym}}
\newcommand{\Id}{\operatorname{I}}
\newcommand{\adj}{\operatorname{adj}} %adjugate matrix
\newcommand{\dd}{\mathrm{d}} % derivative
\newcommand{\tr}{\operatorname{tr}}

\newcommand{\fa}{{\mathfrak a}}

\newcommand{\scal}[1]{\langle #1 \rangle}

\title[Convex Algebraic Geometry of sets with continuous symmetries]{Two results on the Convex Algebraic Geometry of sets with continuous symmetries}

\author[R. G. Bettiol]{Renato G. Bettiol}
\address{\!\!\!\begin{tabular}{lll}
CUNY Lehman College & & CUNY Graduate Center \\
Department of Mathematics & & Department of Mathematics \\
250 Bedford Park~Blvd W & & 365 Fifth Avenue \\
Bronx, NY, 10468, USA & & New York, NY, 10016, USA
\end{tabular}
}
\email{r.bettiol@lehman.cuny.edu}

\author[M. Kummer]{Mario Kummer}
\address{Technische Universit\"at Dresden \newline
\indent Fakult\"at Mathematik \newline
\indent Institut f\"ur Geometrie \newline
\indent Zellescher Weg 12-14 \newline
\indent 01062 Dresden, Germany}
\email{mario.kummer@tu-dresden.de}

\author[R. A. E. Mendes]{Ricardo A. E. Mendes}
\address{University of Oklahoma\newline
\indent Department of Mathematics\newline
\indent 601 Elm Ave\newline
\indent Norman, OK, 73019-3103, USA}
\email{ricardo.mendes@ou.edu}

\allowdisplaybreaks
\numberwithin{equation}{section}

\date{\today}

\begin{document}


\begin{abstract}
We prove two results on convex subsets of Euclidean spaces invariant under an orthogonal group action. 
First, we show that invariant spectrahedra admit an equivariant spectrahedral description, i.e., can be described by an equivariant linear matrix inequality.
Second, we show that the bijection induced by Kostant's Convexity Theorem between convex subsets invariant under a polar representation and convex subsets of a section invariant under the Weyl group preserves the classes of convex semi-algebraic sets, spectrahedral shadows, and rigidly convex sets. 
\end{abstract}
\maketitle

\section{Introduction}

A \emph{spectrahedron} is a subset $S\subset \R^n$ determined by a linear matrix inequality
\[ M(x)=M_0 + x_1 M_1 + \dots + x_n M_n \succeq 0, \]
where $M_j\in\Sym^2(\R^d)$ are symmetric $d\times d$ matrices and $M\succeq0$ means that $M$ is positive-semidefinite. This class of convex sets in $\R^n$ contains all convex polyhedra (as the subclass where all $M_j$'s commute), and has received great attention in recent years in the emerging field of Convex Algebraic Geometry, see, e.g.~\cite{SIAMbook,netzer-plaumann}. 
Spectrahedra play an important role in Optimization and Applied Mathematics as feasible regions for semidefinite programming, a powerful generalization of linear programming on convex polyhedra to find extrema of linear functions, for which efficient interior-point methods are available. In the literature, spectrahedra with symmetries have been studied both for finite symmetry groups, e.g., {\cite{gatermann-parrilo,vallentin}}, and continuous symmetry groups, e.g., \cite{orbitopes, KS22, SS22}. In this note, we focus on the continuous case.

If a spectrahedron $S=\{x\in\R^n : M(x)\succeq0 \}$, with $M\colon\R^n\to \Sym^2(\R^d)$ as above, is invariant under an orthogonal action of a Lie group $\G$ on $\R^n$, it is natural to ask whether there is a $\G$-equivariant spectrahedral description of $S$. An answer entails defining an orthogonal $\G$-representation on $\R^d$ for which $M$ becomes $\G$-equivariant, or, more generally, replacing $M$ with a $\G$-equivariant affine-linear map $\overline M\colon \R^n \to\Sym^2(V)$, where $V$ is some orthogonal $\G$-representation, such that $S=\{x\in\R^n : \overline M(x)\succeq0 \}$. Our first main result answers affirmatively this question:

\begin{mainthm}\label{thm:main-equivdescrp}
If a spectrahedron $S\subset \R^n$ is invariant under the orthogonal action of a compact Lie group $\G$, then there exists a $\G$-representation $V$ and a $\G$-equivariant affine-linear map $\overline M\colon \R^n\to \Sym^2(V)$ such that $S=\{x\in\R^n : \overline M(x)\succeq0 \}$.
\end{mainthm}

The above $\G$-equivariant spectrahedral description for $S$ is built in two steps. First, we define an infinite-dimensional spectrahedral description for $S$, in terms of a $\G$-equivariant linear map $\widetilde M\colon\R^n\to\Sym^2(\cV)$, where $\cV$ is a Hilbert space, see Construction~\ref{con:l2construction}. Second, we find a sufficiently large finite-dimensional subspace $V\subset\cV$ such that $\overline M=\widetilde M|_V$ satisfies $\overline M(x)\succeq 0$ if and only if $M(x)\succeq0$, see Proposition~\ref{thm:finiteequi}. To demonstrate its practical potential, we  carry out this entire construction in three concrete examples, see Examples \ref{ex:constr1}, \ref{ex:constr2}, and \ref{ex:constr3}. For an analogous result concerning finite groups $\mathsf G$, see \cite{Kummer21}.

Our second main result concerns convex sets invariant under a \emph{polar representation}, i.e., an orthogonal action of a compact Lie group $\K$ on $V$ for which there is a vector subspace $\mathfrak a\subset V$ that intersects all $\K$-orbits orthogonally, see Definition~\ref{D:polar}. In this situation, there is a finite group $W\subset \O(\mathfrak a)$, called the Weyl group, such that the orbit spaces $V/\K$ and $\mathfrak a/W$ are isometric. An example is the $\O(n)$-action by conjugation on the space $V=\Sym^2(\R^n)$ of symmetric $n\times n$ matrices.
By the Spectral Theorem, all $\O(n)$-orbits intersect the subspace $\mathfrak a\subset V$ of diagonal matrices orthogonally, and the action of the symmetric group $W=S_n$ permuting entries on $\mathfrak a$ is such that $V/\K\cong \mathfrak a/W$, see Example~\ref{ex:sym}.
By Kostant's Convexity Theorem (see Proposition~\ref{P:1-1} for details), for any polar $\K$-representation $V$, the bijection 
\begin{align}
    \big\{ \K\text{-invariant subsets of } V \big\} &\longleftrightarrow \big\{ W\text{-invariant subsets of } \mathfrak{a} \big\}\nonumber\\
    \mathcal O &\longmapsto \mathcal O\cap \mathfrak{a},\label{eq:bijection-convex} \\
        \K\cdot S &\,\text{\reflectbox{$\longmapsto$}}\; S.\nonumber
\end{align}
maps convex subsets to convex subsets. Thus, it is natural to ask if special classes of invariant convex subsets are preserved by \eqref{eq:bijection-convex}. While we are unable to confirm our expectation that spectrahedra are mapped to spectrahedra (Conjecture~\ref{con:specpolar}), we exhibit classes that contain all spectrahedra which are mapped to one another:

\begin{mainthm}\label{thm:main-polar}
The bijection \eqref{eq:bijection-convex} preserves the following classes of convex subsets:
\begin{enumerate}[\rm (i)]
    % \item Convex sets;
    \item Convex semi-algebraic sets;
    \item Spectrahedral shadows (i.e., affine-linear images of spectrahedra);
    \item Rigidly convex sets.
\end{enumerate}
\end{mainthm}

For a definition of \emph{rigidly convex} set, see Section~\ref{sec:prelim}. According to the Generalized Lax Conjecture~\ref{conj:glc}, all rigidly convex sets are spectrahedra. While this remains a central open problem in Convex Algebraic Geometry, if true, it would imply by Theorem~\ref{thm:main-polar} (iii) that the bijection \eqref{eq:bijection-convex} preserves the class of spectrahedra.
Our proof of Theorem~\ref{thm:main-polar}~(ii) heavily relies on recent work of Kobert and Scheiderer~\cite{KS22}. Earlier work by Sanyal and Saunderson~\cite[Thm.~4.1]{SS22} established Theorem~\ref{thm:main-polar}~(ii) in the special case of the $\O(n)$-action by conjugation on $\Sym^2(\R^n)$. 

In light of \eqref{eq:bijection-convex}, a convex optimization problem whose feasible set $\mathcal O\subset V$ is invariant under a polar $\K$-representation can be reduced to a convex optimization problem with fewer variables, on the $W$-invariant feasible set $\mathcal O\cap \mathfrak a$, see Section~\ref{sec:optimization}. Although this reduction improves efficiency in practice, by Theorem~\ref{thm:main-polar}~(ii), it does not enlarge the class of problems that can be solved with semidefinite programming.

This paper is organized as follows. In Section~\ref{sec:prelim}, we give the necessary background from Convex Algebraic Geometry. Theorem~\ref{thm:main-equivdescrp} is proved in Section~\ref{sec:proofthma}. Relevant background on polar representations is given in Section~\ref{sec:prelimpolar}, and Theorem~\ref{thm:main-polar} is proved in Section~\ref{sec:proofthmb}. In Section~\ref{sec:optimization}, we discuss \eqref{eq:bijection-convex} in the context of convex~optimization.

\subsection{Acknowledgements}
Renato G.~Bettiol is supported by the National Science Foundation CAREER grant DMS-2142575. Mario Kummer is supported by the Deutsche Forschungsgemeinschaft under Grant No. 502861109. Ricardo A.~E.~Mendes is supported by NSF grant DMS-2005373 and the Dodge Family College of Arts and Sciences Junior Faculty Summer Fellowship of the University of Oklahoma.

\section{Preliminaries from Convex Algebraic Geometry}\label{sec:prelim}

In this section, we recall some basic facts about spectrahedra, spectrahedral shadows, and rigidly convex sets; more comprehensive references on the subject are, e.g., \cite{SIAMbook,netzer-plaumann,pastpresentfuture}.
We denote by $\Sym^2(\R^d)$ the vector space of $d\times d$ real symmetric matrices, and by $\Id_d\in \Sym^2(\R^d)$ the identity matrix. Given $A,B\in\Sym^2(\R^d)$, we write $A\succeq B$ to indicate that $A-B$ is positive-semidefinite.

\subsection{Spectrahedra and their shadows}
Let us recall the definitions of spectrahedron and spectrahedral shadow, and make a few initial observations about them.

\begin{Def}
 A set $S \subset \R^n$ is called a \emph{spectrahedron} if there exists an affine-linear map $M\colon\R^n\to\Sym^2(\R^d)$, i.e., $M(x)=M_0+x_1 M_1+\dots +x_n M_n$, where $M_j\in\Sym^2(\R^d)$, such that $S=\{x\in\R^n : M(x)\succeq0\}$. We say that $M$ is \emph{monic} if $M(0)=\Id_d$.
\end{Def}

\begin{rem}\label{rem:monic}
  A spectrahedron can be described by $M(x)\succeq0$ for a monic affine-linear map $M$ if and only if its interior contains the origin, see \cite[Lemma~2.12]{netzer-plaumann}.  
\end{rem}

Differently from polyhedra, images of spectrahedra under affine-linear maps need not be spectrahedra, which motivates the following definition.
  
\begin{Def}
 A set $S\subset\R^n$ is called a \emph{spectrahedral shadow} if there exists a spectrahedron $S'\subset\R^m$ and an affine-linear map $f\colon\R^m\to\R^n$ such that $S=f(S')$.
\end{Def}

Note that if $M\colon \R^n\to\Sym^2(\R^d)$ is a monic affine-linear map, then $p(x)=\det(M(x))$ satisfies $p(0)=1$ and $p$ vanishes on the boundary of the spectrahedron $S=\{x\in\R^n : M(x)\succeq0\}$. Moreover, for any fixed $0\neq w\in\R^n$, the univariate polynomial $t\mapsto p(tw)$ has only real zeros, since the matrices $M(tw)$ are symmetric. 

\begin{Def}
 Let $p\in\R[x]$ be a polynomial and $x_0\in\R^n$ be a point such that $p(x_0)>0$. We denote by $\cC_p(x_0)$ the closure of the connected component of the set
  \begin{equation*}
  \{x\in\R^n : p(x)>0\}
 \end{equation*}
 which contains $x_0$. A set $S\subset\R^n$ is called an \emph{algebraic interior} if there exists $x_0\in\R^n$ and a polynomial $p\in\R[x]$ such that $p(x_0)>0$ and $S=\cC_p(x_0)$.
\end{Def}

Spectrahedra $S$ with nonempty interior $S^\circ\neq\emptyset$ are algebraic interiors \cite[\S2]{pastpresentfuture}; in fact, we will recall in Proposition~\ref{prop:specrigid} that an even stronger property holds.

We now collect some auxiliary results about algebraic interiors that will be useful in later sections to prove Theorem~\ref{thm:main-equivdescrp} in the Introduction. 

\begin{lem}[Lemma~2.1 in \cite{HV07}]\label{lem:defpoly}
 Let $S\subset\R^n$ be an algebraic interior and let $p\in\R[x]$ be a polynomial of minimal degree such that $S=\cC_p(x_0)$ for some $x_0\in\R^n$. Then $p$ is unique (up to multiplication by positive constants) and, for every $q\in\R[x]$ with $S=\cC_q(x'_0)$ for some $x'_0\in\R^n$, we have $q=p\,h$ for some $h\in\R[x]$.
\end{lem}

The polynomial $p$ in Lemma~\ref{lem:defpoly} is called the \emph{defining polynomial} of $S$. 

\begin{lem}\label{lem:propsdefp}
 Let $S\subset\R^n$ be an algebraic interior with defining polynomial $p$. Then $p$ is \emph{square-free}, in the sense that $p$ has no repeated irreducible factors. Moreover, the complex zero set of $p$ is the Zariski-closure of the boundary of $S$.
 In particular, the set of real zeros of $p$ lies Zariski-dense in its complex zero set.
\end{lem}

\begin{proof}
 The fact that $p$ is square-free was observed in the last paragraph of \cite[p.~657]{HV07}. It was also observed in the course of the proof of \cite[Lemma~2.1]{HV07} that the complex zero set of $p$ is the Zariski-closure of the boundary of $S$. This implies that the set of real zeros of $p$ lies Zariski-dense in its complex zero set.
\end{proof}

\begin{cor}\label{cor:zardense}
 Let $S\subset\R^n$ be an algebraic interior with defining polynomial $p$. Let $F_1,\ldots,F_m$ be polynomials whose greatest common divisor is coprime to $p$. Then
 \begin{equation*}
  \partial S\setminus\{a\in\partial S : F_1(a)=\cdots=F_m(a)=0\}
 \end{equation*}
 is Zariski-dense in $\partial S$.
\end{cor}

\begin{proof}
 Consider the decomposition $p=p_1\cdots p_r$ of $p$ into irreducible factors. Let $V_i$ be the complex zero set of $p_i$. Since the complex zero set of $p$ is the Zariski-closure of the boundary of $S$, we have that $S_i=V_i\cap\partial S$ is Zariski-dense in $V_i$. Thus, by Hilbert's Nullstellensatz, every polynomial that vanishes on $S_i$ must be divisible by $p_i$. Therefore, $S_i\setminus\{a\in\partial S : F_1(a)=\cdots=F_m(a)=0\}$ is nonempty open and hence Zariski-dense in $S_i$, which implies the claim.
\end{proof}

Finally, the following result is key in the proof of Theorem~\ref{thm:main-equivdescrp} in the Introduction, to refine infinite-dimensional equivariant spectrahedral representations 
into finite-dimensional ones, see Proposition~\ref{thm:finiteequi}.

\begin{lem}\label{lem:kernidentity}
 Let $M\colon \R^n\to \Sym^2(\R^d)$ be a monic affine-linear map, and $p\in\R[x]$ be the defining polynomial of $S=\{x\in\R^n : {M}(x)\succeq0\}$ as an algebraic interior. There exists a map $\xi\colon\R^n\to \R^d$ whose coordinates are polynomials of degree less than $d$ with greatest common divisor coprime to $p(x)$ such that 
\begin{equation*}
 M(x)\, \xi(x)=q(x)\, p(x)\, v
\end{equation*}
 for some $q\in\R[x]$, and $v\in\R^d$.
\end{lem}

\begin{proof}
 Consider the decomposition $p=p_1\cdots p_r$ of $p$ into irreducible factors. Then, the real zeros of each factor $p_i$ lie Zariski-dense in its complex zero set by Lemma~\ref{lem:propsdefp}. For each $i=1,\ldots,r$ let $m_i\in\N$ be such that $\det(M(x))$ is divisible by $p_i(x)^{m_i}$ but not by $p_i(x)^{m_i+1}$. In particular, we can write
\begin{equation*}
 \det(M(x))=q(x)\, p(x)\, p_1(x)^{m_i-1}\cdots p_r(x)^{m_r-1}
\end{equation*}
for some $q\in\R[x]$.
By \cite[Prop.~4]{2results}, the greatest common divisor of $p_i(x)^{m_i}$ and all entries of the adjugate matrix $\adj(M(x))$ is $p_i(x)^{m_i-1}$. Thus, for a generic $v\in\R^d$, the coordinates of the vector
\begin{equation}\label{eq:xi}
 \xi(x)=\frac{1}{p_1(x)^{m_i-1}\cdots p_r(x)^{m_r-1}}\,\adj(M(x))\, v
\end{equation}
are polynomials of degree less than $d$ whose greatest common divisor is coprime to $p(x)$. Therefore,
\begin{equation*}
 M(x)\, \xi(x)=\frac{\det(M(x))}{p_1(x)^{m_i-1}\cdots p_r(x)^{m_r-1}}\, v = q(x)\, p(x)\, v.\qedhere
\end{equation*}
\end{proof}

\subsection{Rigidly convex sets} 
We now introduce real zero polynomials and rigidly convex sets, see \cite[\S 2]{pastpresentfuture} for details. 

\begin{Def}
 A polynomial $p\in\R[x]$ is called \emph{real zero} with respect to $u\in\R^n$ if $p(u)>0$ and if for every $0\neq w\in\R^n$ the univariate polynomial $p(u+tw)\in\R[t]$ has only real zeros.
\end{Def}

\begin{Def}\label{def:rigid-conv}
 An algebraic interior $S\subset\R^n$ is called \emph{rigidly convex} if its defining polynomial $p$ is real zero with respect to some interior point $u\in S^\circ$.
\end{Def}

The following foundational result from \cite{HV07} shows that the word ``some'' in Definition~\ref{def:rigid-conv} can be replaced with ``every'', and justifies the alluded convexity.

\begin{prop}[\cite{HV07}, see Theorem~6.3 in \cite{SIAMbook}]\label{prop:rzprops}
 Let $S\subset\R^n$ be rigidly convex and $p$ be its defining polynomial as algebraic interior. Then:
 \begin{enumerate}[\rm (i)]
  \item $p$ is real zero with respect to every $u\in S^\circ$.
  \item $S$ is convex.
 \end{enumerate}
\end{prop}

From Proposition~\ref{prop:rzprops} and the definition of the real zero property, we conclude:

\begin{cor}
 Let $S\subset\R^n$ be rigidly convex and $p$ be its defining polynomial as an algebraic interior. Then $p(u)>0$ for every $u\in S^\circ$.
\end{cor}

We now prove a simple criterion to check if nested rigidly convex sets coincide.

\begin{cor}\label{cor:boundaryvanish}
 Let $S_1 \subset S_2 \subset \R^n$ be two rigidly convex sets. If the defining polynomial $p$ of $S_2$ vanishes completely on $\partial S_1$, then $S_1=S_2$.
\end{cor}

\begin{proof}
 Assume that there is $a\in S_2\setminus S_1$. Let $b\in S_1^\circ$ and consider the line segment $[a,b]\subset\R^n$ between $a$ and $b$. Then we have that $[a,b]\cap S_1$ is of the form $[c,b]$ for some $c\in(a,b)\cap \partial S_1$ because $S_1$ is closed and convex. Since $S_2$ is a closed convex set with a nonempty interior, every point in its boundary is contained in a proper face \cite[Cor.~2.8]{barvinok}. Thus, we have that $c\in \partial S_1\cap S_2^\circ$. However, by Proposition~\ref{prop:rzprops} (i), the polynomial $p$ does not vanish on $c$.
\end{proof}

The main motivation to study rigidly convex sets comes from the following result.

\begin{prop}[\S2.2 in \cite{pastpresentfuture}]\label{prop:specrigid}
 If $S\subset\R^n$ is a spectrahedron with nonempty interior, then $S$ is rigidly convex.
\end{prop}

Whether or not the converse of Proposition~\ref{prop:specrigid} holds is an open question:

\begin{con}[Generalized Lax Conjecture~\cite{HV07}]\label{conj:glc}
    Every rigidly convex set is a spectrahedron.
\end{con}

\section{Equivariant descriptions of invariant spectrahedra}\label{sec:proofthma}

Throughout the section, we let $\G$ be a compact Lie group acting orthogonally on $\R^n$, and we fix a Haar measure $\mu$ on $\G$. We begin with an elementary example:

\begin{ex}\label{ex:inv-subspace}
Any $\G$-invariant linear subspace $W\subset\R^n$ admits a $\G$-equivariant spectrahedral representation. Indeed, the  orthogonal direct sum decomposition $\R^n=W\oplus W^\perp$ yields a $\G$-equivariant orthogonal projection $\pi\colon \R^n\to W^\perp$, and
\begin{equation}\label{eq:inv-subspace-specrep}
    W=\left\{x \in \R^n : M(x)=\begin{pmatrix}
        0 & \pi(x) \\ \pi(x)^T & 0
    \end{pmatrix} \succeq 0 \right\},
\end{equation}
where $M\colon \R^n\to \Sym^2(W^\perp\oplus \R)$ is $\G$-equivariant with respect to the $\G$-repre\-sentation where $g\in\G$ acts on $\Sym^2(W^\perp\oplus \R)$ by conjugation with 
$\begin{pmatrix}
        g &  0 \\ 0 & 1
    \end{pmatrix}$. In particular, the trivial spectrahedron $W=\{0\}$ admits a $\G$-equivariant spectrahedral representation, given by \eqref{eq:inv-subspace-specrep} with $\pi(x)=x$. 
    
    Moreover, any $\G$-invariant affine subspace $S\subset\R^n$ admits a $\G$-equivariant spectrahedral description, since $S=W+v$ where $W$ is a $\G$-invariant linear subspace and $v\in S$ is a fixed vector, e.g., we may take $v$ to be the center of mass of the $\G$-orbit of some $s\in S$. The composition of \eqref{eq:inv-subspace-specrep} and the translation by $-v$, which is $\G$-equivariant, gives the spectrahedral representation $S=\{x\in\R^n :M(x-v)\succeq 0\}$.
\end{ex}

Given a nontrivial $\G$-invariant spectrahedron $S=\{x\in\R^n : M(x)\succeq0\}$, we shall assume without loss of generality that $0\in\R^n$ is an interior point of $S$ and that $M$ is \emph{monic}, i.e., $M(0)=\Id_d$. Indeed, if $S^\circ=\emptyset$, then we may replace $\R^n$ with the affine hull $\operatorname{aff}(S)$ of $S$, where the relative interior of $S$ is nonempty. If $S\subset \operatorname{aff}(S)$ admits a $\G$-equivariant spectrahedral description, then so does $S\subset\R^n$ by Example~\ref{ex:inv-subspace}. Moreover, we may assume $0\in S^\circ$ up to translating by an interior point in $S$ which is fixed by $\G$, e.g., the center of mass of the $\G$-orbit of some $s\in S^\circ$. Finally, we may assume %that $M(0)\succ0$ by \cite[Lemma 2.7]{BKM-Siaga}, and hence 
that $M(x)$ is monic (cf.~Remark~\ref{rem:monic}) by replacing it with the $r\times r$ upper left submatrix of $P^T M(x)P$, where $P$ is invertible and such that 
\begin{equation*}
 P^T M(0) P=\begin{pmatrix}
     \Id_r&0\\ 0&0
 \end{pmatrix}.
\end{equation*}
So, henceforth in this section, consider a spectrahedron $S=\{x\in\R^n : M(x)\succeq0\}$, where the affine-linear map $M\colon \R^n\to\Sym^2(\R^d)$ is monic, and denote by
\begin{equation}\label{eq:sgformula}
    S_\G:=\{ x\in\R^n :  M(gx)\succeq0 \; \text{ for all } g\in \G \}
\end{equation}
the largest $\G$-invariant subset of $S$. Clearly, $S=S_\G$ if and only if $S$ is $\G$-invariant.
We now construct an infinite-dimensional spectrahedral description for $S_\G$.

\begin{cons}\label{con:l2construction}
 Let $\mathcal{V}=L^2(\G,\R^d)$ be the space of square-integrable functions from $\G$ to $\R^d$ with respect to $\mu$, which is a Hilbert space with inner product
\begin{equation}\label{eq:ip}
\scal{\varphi,\psi}=\int_{\G}\varphi(g)^T \psi(g) \;\dd \mu(g),
\end{equation}
where elements of $\R^d$ are viewed as column vectors.
The group $\G$ acts on $\mathcal{V}$ on the left, preserving the above inner product, via the left regular representation
\begin{equation}\label{eq:g-action-phi}
 (g\cdot \varphi)(h)=\varphi(g^{-1}h)
\end{equation}
for all $g,h\in \G$ and $\varphi\in\cV$. Let $\Sym^2(\mathcal{V})$ denote the space of symmetric bilinear maps $\mathcal{V}\times\mathcal{V}\to\R$ and let $\G$ act (on the left) on $\Sym^2(\mathcal{V})$ by
\begin{equation}\label{eq:g-action-Q}
 (g\cdot Q)(\varphi,\psi)=Q(g^{-1}\cdot \varphi, g^{-1}\cdot \psi)
\end{equation}
for all $Q\in\Sym^2(\cV)$, $g\in \G$, and $\varphi,\psi\in\cV$.

We say $Q\in \Sym^2(\mathcal{V})$ is \emph{positive-semidefinite}, and write $Q\succeq 0$, when $Q(\varphi,\varphi)\geq 0$ for all $\varphi\in \mathcal{V}$.
We associate to the monic affine-linear map $M\colon \R^n\to \Sym^2(\R^d)$ the monic affine-linear $\G$-equivariant map $\widetilde{M}\colon \R^n\to\Sym^2(\mathcal{V})$ given by
\begin{equation}\label{eq:tildeM}
 \widetilde{M}(x)(\varphi,\psi):=\int_{\G} \varphi(g^{-1})^T \, M(gx)\, \psi(g^{-1}) \; \dd \mu(g).
\end{equation}
\end{cons}

\begin{prop}\label{prop:prop}
If $M\colon \R^n\to \Sym^2(\R^d)$ is a monic affine-linear map, then $\widetilde{M}\colon \R^n\to\Sym^2(\mathcal{V})$ given by \eqref{eq:tildeM} is a monic affine-linear $\G$-equivariant map and
  \begin{equation*}
  S_\G = \{x\in\R^n : \widetilde{M}(x)\succeq0\}. %\{ x\in\R^n : M(gx)\succeq0, \forall g\in \G\}.
  \end{equation*}
\end{prop}

\begin{proof}
By construction, $\widetilde{M}$ is $\G$-equivariant, affine-linear, and $\widetilde M(0)(\varphi,\psi)=\langle \varphi,\psi\rangle$. Let $x\in\R^n$ be such that $M(gx)\succeq 0$ for every $g\in \G$. Then $\widetilde{M}(x)(\varphi,\varphi)$ is the integral of a nonnegative function on $\G$ for every $\varphi\in\mathcal{V}$, and hence nonnegative.

For the reverse inclusion, suppose $M(gx)\not \succeq 0$ for some $x\in\R^n$ and $g\in \G$. Choose $v\in\R^d$ such that $v^TM(gx) v<0$. By continuity, there is an open neighborhood $U\subset \G$ of $g$ such that $v^TM(hx)v<0$ for all $h\in U$. Let $\rho\colon \G\to\R$ be a continuous test function with $\rho(g)=1$ and whose support is contained in $U$. Define $\varphi\in\mathcal{V}$ by $\varphi(h)=\rho(h^{-1})v$ for all $h\in \G$. Then $\widetilde{M}(x)\not\succeq 0$, because
\begin{equation*}
 \widetilde{M}(x)(\varphi,\varphi) =\int_{\G} \rho(h)^2\, v^T M(hx) v \;\dd \mu(h)=\int_{U} \rho(h)^2 \, v^TM(hx) v \;\dd \mu(h)<0.\qedhere
\end{equation*}
\end{proof}

We now examine when there exists a \emph{finite-dimensional} $\G$-invariant subspace $V\subset\cV$ such that $\overline M:=\widetilde M|_V$ satisfies
$\overline M(x)\succeq0$ if and only if $\widetilde M(x)\succeq0$, yielding a $\G$-equivariant spectrahedral description of $S_\G$. In general, one cannot expect $S_\G$ to be a spectrahedron, in fact not even a spectrahedral shadow, as illustrated by:

\begin{ex}\label{ex:S_G}
    Set $d=1$ and $n=\binom{m+k}{m}$. We identify $\R^n$ with the vector space $\R[x_1,\ldots,x_m]_{\leq k}$ of all polynomials of degree at most $k$ in $m$ variables, endowed with the natural action of the orthogonal group $\G=\O(m)$. Consider the monic affine-linear map $M\colon\R^n\to \R$, given by $M(P) = 1-P(1,0,\ldots,0)$, and the halfplane (which is a spectrahedron) $S=\{P\in\R^n : M(P)\leq0\}$. 
    As a consequence of \eqref{eq:sgformula} and Proposition~\ref{prop:prop}, the subset $S_\G=\{P\in\R^n : \widetilde{M}(P)\succeq0\}$ consists of all polynomials $P$ on $\R^m$ of degree at most $k$ that are bounded by $1$ on the unit sphere in $\R^m$. By Scheiderer~\cite{shadows}, this set is not a spectrahedral shadow for large enough $m$ and $k$. 
    
    Let us describe a similar example in the vector space $\Sym^2_b(\wedge^2\R^m)$ of algebraic curvature operators $R\colon\wedge^2\R^m\to\wedge^2\R^m$, $m\geq5$, see \cite{BKM-Siaga} for details. Set $d=1$ and $n=\frac{1}{12}m^2(m^2-1)$, and consider the action of $\G=\O(m)$ on $\Sym^2_b(\wedge^2\R^m)\cong\R^n$. The map $M(R)=1+\sec_R(e_1\wedge e_2)$ is a monic affine-linear map, that gives rise to the spectrahedron $S=\{R\in\R^n : M(R)\geq0\}$. Since the $\G$-action is transitive on the oriented Grassmannian of $2$-planes in $\R^m$, we have that $S_\G=\mathfrak R_{\sec\geq-1}(m)$ is the set of curvature operators with sectional curvature bounded below by $-1$, which is not a spectrahedral shadow for $m\geq5$ by \cite[Thm.~A]{BKM-Siaga}.
\end{ex}

In light of the reductions discussed in the beginning of the section, the following result implies Theorem~\ref{thm:main-equivdescrp} in the Introduction. 

\begin{prop}\label{thm:finiteequi}
 Let $M\colon \R^n\to \Sym^2(\R^d)$ be a monic affine-linear map, and
 \begin{equation*}
   S=\{x\in\R^n : {M}(x)\succeq0\}
 \end{equation*}
 be the corresponding spectrahedron. If $S$ is $\G$-invariant, i.e., $S=S_\G$, then there exists a finite-dimensional $\G$-invariant linear subspace $V\subset\cV$ such that $\overline M:=\widetilde M|_V$ satisfies $\overline M(x)\succeq0$ if and only if $\widetilde M(x)\succeq0$. 
In particular, $S$ admits a $\G$-equivariant spectrahedral description $S=\{x\in\R^n : \overline{M}(x)\succeq0\}$.
\end{prop}

\begin{proof}
By Proposition~\ref{prop:prop} and the assumption that $S$ is $\G$-invariant, we have that $S=\{x\in \R^n : \widetilde M(x)\succeq0\}$. Let $p$ be the defining polynomial of $S$ as an algebraic interior, and let $\xi\colon\R^n\to\R^d$, $q\in\R[x]$ and $v\in\R^d$ be as in Lemma~\ref{lem:kernidentity}. Let $W$ be the span (in the vector space of polynomials of degree less than $d$) of the coordinates of $\xi(g^{-1}x)$ for all $g\in \G$. Choose a basis $F_1(x),\ldots,F_m(x)$ of $W$ and let
\begin{equation*}
 F(x)=\big(F_1(x),\ldots,F_m(x)\big)^T.
\end{equation*}
For all $g\in \G$, there is a unique $d\times m$ matrix $B(g)$ such that $B(g)\, F(x)=\xi(g^{-1}x)$. We first prove that, as functions on $\G$, the columns of $B(g)$ are linearly independent elements of $\cV$. 
Suppose this is not the case: assume without loss of generality that 
\begin{equation*}
    B^1(g)=\sum_{j=2}^m c_j B^j(g)
\end{equation*}
for some $c_2,\ldots,c_m\in\R$, where $B^j(g)$ is the $j$th column of $B(g)$. By construction,
\begin{equation*}
    \xi(g^{-1}x)=\sum_{j=1}^m F_j(x) B^j(g)=\sum_{j=2}^m \big(F_j(x)+c_jF_1(x)\big) B^j(g)
\end{equation*}
for all $g\in\G$ and $x\in\R^n$, but this means that already $\big(F_j(x)+c_jF_1(x)\big)_{j=2,\ldots,m}$ spans $W$, contradicting our choice of the $F_j(x)$.

Now we prove that the finite-dimensional linear subspace $V$ of $\mathcal{V}$ that is spanned by the columns of $B(g)$ fulfills the desired property; namely, $\overline M:=\widetilde M|_V$ satisfies $\overline M(x)\succeq0$ if and only if $\widetilde M(x)\succeq0$. 
Let $\overline S:=\{ a\in\R^n : \overline{M}(a) \succeq0\}$; we have to show that $\overline S=S$. The inclusion $S \subset \overline S$ is clear, since $\widetilde{M}(a) \succeq0$ implies $\overline{M}(a) \succeq0$. Therefore, by Corollary~\ref{cor:boundaryvanish} and Proposition~\ref{prop:specrigid}, it suffices to prove that the defining polynomial $\overline p$ of $\overline S$ vanishes on the boundary of $S$. By Corollary~\ref{cor:zardense}, it further suffices to prove that $\overline p$ vanishes on $U=\{a\in\partial S : \xi(a)\neq0\}$. By construction, $F(a)\neq0$ for all $a\in U$. Let $a\in U$ and consider $\psi\in V$ defined by
\begin{equation*}
 \psi(g):=B(g)\, F(a).
\end{equation*}
Since $F(a)\neq0$, we have $\psi\neq0$. For every $\varphi\in V$, we then have, by construction,
\begin{align*}
 \widetilde{M}(a)(\varphi,\psi)&=\int_{\G} \varphi(g^{-1})^T \, M(ga)\,\psi(g^{-1}) \;\dd \mu(g)\\
 &=\int_{\G} \varphi(g^{-1})^T \, M(ga)\, \xi(ga) \;\dd \mu(g)\\
 &=\int_{\G} q(ga) p(ga) (\varphi(g^{-1})^T v) \;\dd \mu(g).
\end{align*}
Since $a\in\partial S$, we have $p(a)=0$. Since $S$ is $\G$-invariant, we also have $ga\in\partial S$ for all $g\in \G$, and hence $p(ga)=0$. This shows that $\widetilde{M}(a)(\varphi,\psi)=0$, which means that $\overline M(a)=\widetilde{M}(a)|_V$ has nonzero kernel and therefore $a$ must lie on the boundary of $\overline S$ because $\overline{M}(x)\succ0$ is strictly positive-definite for $x$ in the interior of $\overline S$. Since $a$ is in the boundary of $\overline S$, we have, in particular, that $\overline p$ vanishes on $a$. Hence $\overline p$ vanishes on $U$, which proves the claim.
\end{proof}

\begin{rem}\label{rem:conn}
The above proof and basic facts about Lie groups imply that if $S$ is $\G_0$-invariant, where $\G_0\subset \G$ is the connected component of the identity, then there exists a finite-dimensional $\G$-invariant linear subspace $V\subset\cV$ such that $\overline M:=\widetilde M|_V$ satisfies $\overline M(x)\succeq0$ if and only if $\widetilde M(x)\succeq0$. In particular, $S_\G$ admits a $\G$-equivariant spectrahedral description $S_\G=\{x\in\R^n : \overline{M}(x)\succeq0\}$.
\end{rem}

\begin{ex} \label{ex:constr1}
Let $S\subset\R^2$ be the unit disk, which is the spectrahedron formed by $x=(x_1,x_2)\in\R^2$ that satisfy the linear matrix inequality $M(x)\succeq0$, where
\begin{equation}\label{eq:M-disk}
M\colon \R^2\to \Sym^2(\R^2), \qquad 
 M(x)=\begin{pmatrix}
 1+x_1& -x_2\\ -x_2 & 1-x_1
 \end{pmatrix}.
\end{equation}
Clearly, $S$ is invariant under the action of $\G=\SO(2)$ on $\R^2$ given by the rotations
\begin{equation}\label{eq:gtO2}
g(t)=\begin{pmatrix}
\cos t &-\sin t \\ \sin t& \cos t\end{pmatrix}, \qquad t\in[0,2\pi].
\end{equation}
There does not exist any $\G$-action on $\Sym^2(\R^2)$ induced by an action on $\R^2$ that makes \eqref{eq:M-disk} $\G$-equivariant. % but it is true that $M(g(2t)(x))=g(-t) M(x) g(t)$!
Indeed, $g(\pi)=-\Id_2$ is in the kernel of any such action and $M(-x)\neq M(x)$ for all $x\neq 0$. We now apply our Construction~\ref{con:l2construction} to obtain a $\G$-equivariant spectrahedral description $\overline M\colon\R^2\to\Sym^2(V)$ for $S$; as it turns out, we can use $V\cong\R^3$.

The algebraic boundary of $S$ is the zero set of $p(x)=1-x_1^2-x_2^2$, and from \eqref{eq:xi}, setting $v=(1,0)^T$, we obtain that $\xi(x)=(1-x_1,x_2)^T$ satisfies $M(x)\xi(x)=p(x)v$. Thus, the vector space $W$ from the proof of Proposition~\ref{thm:finiteequi} is the $3$-dimensional space of polynomials in $\R^2$ of degree at most $1$. We set $F(x)=(1,-x_1,-x_2)^T$, and 
$$B(g(t))=\begin{pmatrix}
1& \cos t &\sin t \\
0& \sin t & -\cos t
\end{pmatrix},$$
which satisfies $B(g(t)) F(x)=\xi(g(-t)x)=\xi(g(t)^{-1}x)$. 
We fix the Haar measure $\dd\mu(g(t))=\frac{1}{2\pi}\dd t$ on $\G$ and let $V\subset L^2(\G,\R^2)$ be the subspace spanned by the columns of $B(g(t))$, which form an orthonormal basis with respect to \eqref{eq:ip}.

By \eqref{eq:tildeM}, the Gram matrix of $\overline M(x):=\widetilde{M}(x)|_V$ with respect to the above basis~is
$$\frac{1}{2\pi}\int_0^{2\pi}
\begin{pmatrix}
 1+x_1 \cos t-x_2 \sin t & x_1+\cos t & x_2-\sin t \\
 x_1+\cos t & 1+ x_1 \cos t+x_2 \sin t & x_2 \cos t-x_1 \sin t \\
 x_2-\sin t & x_2 \cos t -x_1 \sin t & 1-x_1 \cos t-x_2 \sin t  
\end{pmatrix}\;\dd t,$$
which yields the following $\G$-equivariant spectrahedral description of the unit disk:
$$
S=\left\{x\in\R^2 : \overline M(x)= \begin{pmatrix}
   1&x_1&x_2\\x_1&1&0\\x_2&0&1
  \end{pmatrix}\succeq0\right\}.$$

Since the left regular representation \eqref{eq:g-action-phi} of $\G$ on $\cV=L^2(\G,\R^2)$ is given by $(g(s)\cdot \varphi)(t)=\varphi(t-s)$, the corresponding $\G$-representation \eqref{eq:g-action-Q} on $\Sym^2(V)\subset \Sym^2(\cV)$ is via the following matrix conjugation:
\begin{equation*}
    g(s)\cdot Q=
    \begin{pmatrix}
 1 & 0 & 0 \\
 0 & \cos s & \sin s \\
 0 & -\sin s & \cos s \\
    \end{pmatrix}
    \, Q \,
    \begin{pmatrix}
 1 & 0 & 0 \\
 0 & \cos s & -\sin s \\
 0 & \sin s & \cos s \\
    \end{pmatrix}, \qquad Q\in\Sym^2(V).
\end{equation*}
From the above, one easily sees that $\overline M(g(s)x)=g(s)\cdot \overline M(x)$, and hence $\overline{M}$ is indeed $\G$-equivariant.
\end{ex}

\begin{ex} \label{ex:constr2}
Let $\R[x,y]_4$ be the vector space of binary quartics, and consider the convex cone $C\subset\R[x,y]_4^*$ given by functionals $\ell\colon \R[x,y]_4\to\R$ of the form $  \ell(p) = \int_{S^{1}} p(x) \; \dd \mu(x)$, where $\mu$ is a probability measure on $S^1\subset \R^2$. Clearly, $C$ is invariant under the $\mathsf O(2)$-action $ (g\cdot \ell)(p) = \ell(p\circ g)$, where $g\in \mathsf O(2)$, $\ell\in C$. 
Fix the basis $(x^4, x^3 y, x^2 y^2, x y^3, y^4)$ of $\R[x,y]_4$, and write dual coordinates $(\ell_1,\dots,\ell_5)$ on $\R[x,y]_4^*$. 
Setting $\mathfrak m=(x^2, xy, y^2)$, there is a spectrahedral description of $C$ as the elements $\ell\in \R[x,y]_4^*$ such that $M(\ell)=\ell(\mathfrak m^T \mathfrak m)\succeq0$, where $\ell$ is applied entrywise:
\begin{equation*}
  C
  = \left\{(\ell_1,\dots,\ell_5)\in\R^5 : M(\ell)=\begin{pmatrix}
 \ell_1 & \ell_2 & \ell_3 \\
 \ell_2 & \ell_3 & \ell_4 \\
 \ell_3 & \ell_4 & \ell_5 \\
  \end{pmatrix}\succeq0  \right\}.
\end{equation*}
The functional $\ell_0(p)=\frac{1}{2\pi}\int_{0}^{2\pi} p(\cos t,\sin t)\;\dd t$ given by integration with respect to the uniform probability measure on $S^1$ has coordinates $\ell_0=(\frac38, 0, \frac18, 0, \frac38)$. Translating $\ell_0$ to the origin and performing $M\mapsto P^T M P$ for an appropriate invertible matrix $P$, we obtain the monic affine-linear map $\overline M\colon\R^5\to\Sym^2(\R^3)$ given by
\begin{equation*}
    \overline M(\ell)=\begin{pmatrix}
 \ell_1+2 \ell_3+\ell_5+1 & \sqrt{2} (\ell_5-\ell_1) & -2 \sqrt{2} (\ell_2+\ell_4) \\
 \sqrt{2} (\ell_5-\ell_1) & 2 \ell_1-4 \ell_3+2 \ell_5+1 & 4 (\ell_2-\ell_4) \\
 -2 \sqrt{2} (\ell_2+\ell_4) & 4 (\ell_2-\ell_4) & 8 \ell_3+1 \\
  \end{pmatrix}.
\end{equation*}
While Construction~\ref{con:l2construction} leaves $\overline M\colon\R^5\to\Sym^2(\R^3)$ unchanged, it also leads to an $\mathsf O(2)$-action on $\Sym^2(\R^3)$ that makes it $\mathsf O(2)$-equivariant. Namely, if $g(t)\in\mathsf O(2)$ as in \eqref{eq:gtO2} acts on $Q\in \Sym^2(\R^3)$ by $g(t)\cdot Q= \operatorname{diag}(1,g(2t)) \, Q\, \operatorname{diag}(1,g(-2t))$, and acts on $\ell\in \R^5$ by keeping track of the action on the above basis of $\R[x,y]_4$ when it acts by the defining representation on $(x,y)\in\R^2$, then $\overline M$ is $\mathsf O(2)$-equivariant.
\end{ex}

\begin{ex}\label{ex:constr3}
The set of positive-semidefinite $2\times 2$ Hermitian matrices
\begin{equation}\label{eq:herm}
    \begin{pmatrix}
        a_{11} & a_{12}+i \, b_{12}\\
        a_{12} -i \, b_{12} & a_{22}
    \end{pmatrix}
\end{equation}
forms a spectrahedron $S\subset \R^4$ that can be described by the linear matrix inequality $M(a_{11},a_{12},a_{22},b_{12})\succeq0$, where
\begin{equation}\label{eq:M-psdH}
M\colon \R^4\to \Sym^2(\R^3), \quad 
 M(a_{11},a_{12},a_{22},b_{12})=\begin{pmatrix}
 a_{11} & a_{12} & 0 \\
 a_{12} & a_{22} & b_{12} \\
 0 & b_{12} & a_{11} 
 \end{pmatrix},
\end{equation}
is the leading principal minor of order $3$ of the real $4 \times 4$ symmetric matrix corresponding to \eqref{eq:herm}.
The Lie group $\G=\SU(2)$ consists of the $2\times 2$ matrices
\begin{equation*}
    g(x,y,s,t)=\begin{pmatrix}
x+i\,y & -s+i\,t \\
        s+i\,t & \phantom{+}x-i\,y
\end{pmatrix}, \quad x,y,s,t\in\R, \; x^2+y^2+s^2+t^2=1,
\end{equation*}
and acts by conjugation on \eqref{eq:herm} leaving invariant the spectrahedron $S$. Since $\det M(a_{11},a_{12},a_{22},b_{12})$ is not $\G$-invariant, the spectrahedral description \eqref{eq:M-psdH} is not $\G$-equivariant for any $\G$-representation on $\R^3$. Our Construction~\ref{con:l2construction} yields an equivariant $\overline M\colon\R^4\to\Sym^2(V)$, with $V\cong\R^4$, as follows.

In order to work with monic matrix polynomials, we shift \eqref{eq:M-psdH} by $\Id_3$. The defining polynomial of the shifted spectrahedron as an algebraic interior is
$$p(a_{11},a_{12},a_{22},b_{12})=1 + a_{11} - a_{12}^2 + a_{22} + a_{11} a_{22} - b_{12}^2,$$
and $\xi=(-a_{12}, 1+a_{11}, -b_{12})^T$ satisfies $(M+\Id_3)\xi = p\, v$ where $v=(0,1,0)^T$.
As in the proof of Proposition~\ref{thm:finiteequi}, analyzing how $\xi$ behaves under the $\G$-action, we set
\begin{align*}
    F &=\left(1 + \tfrac12 a_{11} +\tfrac12 a_{22},\, -a_{12},\, -\tfrac12 a_{11} + \tfrac12 a_{22},\, b_{12}\right)^T,\\
B(g(x,y,s,t))&=\begin{pmatrix}
 0 & 2 x^2+2 t^2-1 & 2 t y-2 s x & -2 x y-2 s t \\
 1 & -2 s x-2 t y & 2 s^2+2 t^2-1 & 2 s y-2 t x \\
 0 & 2 s t-2 x y & 2 s y+2 t x & 2 y^2+2 t^2-1 \\
\end{pmatrix},
\end{align*}
and let $V\subset L^2(\G,\R^3)$ be the subspace spanned by the columns of $B(g)$, which form an orthonormal basis with respect to \eqref{eq:ip}. 
Here, the Haar measure on $\G$ is the probability measure $\dd\mu(g)=\frac{1}{2\pi^2}\dd\operatorname{vol}$ proportional to the Riemannian volume form of the unit round metric on $\G\cong S^3$.
Using \eqref{eq:tildeM}, we compute the Gram matrix of $\overline M:=\widetilde{M}|_V$ in this basis and, shifting back by $-\Id_4$, we obtain a $\G$-equivariant spectrahedral description of $S$ as the set of $(a_{11},a_{12},a_{22},b_{12})\in\R^4$ such that
\begin{equation*}
    \overline M=\begin{pmatrix}
\frac12 a_{11}+ \frac12 a_{22} & a_{12} &\frac12 a_{11}- \frac12 a_{22} & -b_{12} \\
 a_{12} & \frac12 a_{11}+ \frac12 a_{22} & 0 & 0 \\
\frac12 a_{11}- \frac12 a_{22} & 0 & \frac12 a_{11}+ \frac12 a_{22} & 0 \\
-b_{12} & 0 & 0 & \frac12 a_{11}+ \frac12 a_{22}
\end{pmatrix}\succeq0
\end{equation*}
Note that the diagonal entries are a multiple of the trace of \eqref{eq:herm}, hence invariant under the $\G$-action on $S$, as is the first column of $B_g$ under the $\G$-action on $V$.
\end{ex}

\begin{rem}
Starting from a spectrahedral description of size $d$ and $n$ variables, that is, with an affine-linear map $M\colon\R^n\to\Sym^2(\R^d)$, the equivariant spectrahedral description $\overline M\colon \R^n\to\Sym^2(\R^{\overline d})$ obtained via Theorem \ref{thm:main-equivdescrp} has size $\overline d\leq \binom{n+d-1}{n}$. This bound follows from the proof of Proposition~\ref{thm:finiteequi}, as $B(g)$ has $m=\dim W$ columns and $W\subset \R[x_1,\dots,x_n]_{< d}$. Although this bound is sharp in Example~\ref{ex:constr1}, the actual size $\overline d$ can be significantly smaller, e.g., in Example~\ref{ex:constr2}, where $\overline d=3$ and $\binom{n+d-1}{n}=21$, and in Example~\ref{ex:constr3}, where $\overline d=4$ and $\binom{n+d-1}{n}=15$.
\end{rem}

As illustrated by Example \ref{ex:S_G}, one cannot expect that $S_\G$ is a spectrahedron, even if $S$ is. The following  partial converse to Proposition~\ref{thm:finiteequi} shows that if $S_\G$ is a spectrahedron, which, moreover, is given by linear matrix inequalities coming from Construction~\ref{con:l2construction}, then, in fact, $S$ was already $\G$-invariant (assuming $\G$ is connected, cf.~Remark~\ref{rem:conn}).

\begin{prop}\label{prop:converse}
Let $M\colon \R^n\to \Sym^2(\R^d)$ be a monic affine-linear map, and $S=\{x\in\R^n : {M}(x)\succeq0\}$ be the corresponding spectrahedron. Assume there is a finite-dimensional $\G$-invari\-ant subspace $V\subset \cV$ such that $\overline M:=\widetilde M|_V$ satisfies
 $\overline M(x)\succeq0$ if and only if $\widetilde M(x)\succeq0$. 
If $\G$ is connected, then $S$ is $\G$-invariant.
\end{prop}

\begin{proof} 
By Proposition~\ref{prop:prop}, we have that $S_\G=\{x\in\R^n : \widetilde{M}(x)\succeq0\}$, so it suffices to show that $S_\G=S$, see \eqref{eq:sgformula}.
By Proposition~\ref{prop:specrigid} and Corollary~\ref{cor:boundaryvanish}, it suffices to prove $\partial S_\G\subset \partial S$. 
Note that $S_\G$ is a spectrahedron, since we are assuming $S_\G = \{ x\in\R^n : \overline M(x) \succeq0\}$. Thus, given $x\in \partial S_\G$, the kernel of $\overline{M}(x)$ is nonzero, so there exists $\varphi\neq 0$ in $V$ such that $\widetilde{M}(x)(\varphi,\varphi)=0$, i.e.,
\begin{equation*}
  \int_{\G} \varphi(g^{-1})^T \, M(gx)\, \varphi(g^{-1}) \;\dd \mu(g)=0.  
\end{equation*}
Since we have $M(gx)\succeq 0$ for all $g\in \G$ by \eqref{eq:sgformula}, it follows that
\begin{equation*}
  \varphi(g^{-1})^T \, M(gx) \,\varphi(g^{-1})=0  
\end{equation*}
for almost all $g\in \G$. But, since $\varphi\in V$ and $V$ is finite-dimensional and $\G$-invariant, it follows that $\varphi\colon \G\to \R^d$ is an analytic function (see \cite[III.4.7, Ex.~2, p.~138]{BtD}, where functions like $\varphi$ are called \emph{representative functions}). Since we are assuming $\G$ to be connected, it follows that the zero set of $\varphi$ has measure zero. Thus, for almost all $g$, the matrix $M(gx)$ has nonzero kernel. This implies that $gx\in\partial S$ for almost all $g$, because $M$ is positive-definite in the interior of $S$. Since $\partial S$ is closed, we have $gx\in \partial S$ for all $g\in \G$ and hence $x\in \partial S$.
\end{proof}

The statement of Proposition~\ref{prop:converse} becomes false if we drop the assumption that $\G$ is connected. Indeed, if $\G$ is finite, then $\cV$ itself is finite-dimensional.

\section{Preliminaries from polar representations}
\label{sec:prelimpolar}

We now recall basic facts on polar representations, see \cite{PT87, Dadok85} for details.

\begin{Def}\label{D:polar}
    Let $\K$ be a compact group, and let $V$ be a real finite-dimensional vector space with inner product, on which $\K$ acts by orthogonal transformations. We say $V$ is a \emph{polar} $\K$-representation if it admits a \emph{section}, that is, a vector subspace $\mathfrak{a}\subset V$ which meets every $\K$-orbit orthogonally. 
\end{Def}

Given a polar $\K$-representation $V$ with section $\mathfrak{a}\subset V$, there is a finite subgroup $W\subset \O(\mathfrak{a})$, called the \emph{Weyl group} (sometimes, and perhaps more appropriately,  called the \emph{generalized Weyl group}), with the property that, for every $v\in\mathfrak{a}$, the $W$-orbit $W\cdot v$ coincides with the intersection of $\mathfrak{a}$ with the $\K$-orbit $\K\cdot v$. In other words, the inclusion $\mathfrak{a}\hookrightarrow V$ induces a bijection $\mathfrak{a}/W\to V/\K$ of orbit spaces.

\begin{rem}
If $\K_0\subset \K$ is the connected component of the identity, then the $\K$-representation $V$ is polar if and only if its restriction to $\K_0$ is polar, and their sections are the same. Moreover, the Weyl group of the latter is easily seen to be a subgroup of the Weyl group of the former.
\end{rem}

Polar representations $V$ of connected groups $\K$ have been classified in \cite{Dadok85}, from which it follows  that they are all orbit-equivalent to (that is, have the same orbits as) the isotropy representation of a symmetric space (\cite[Prop.~6]{Dadok85}, see also \cite[Prop.~1.2]{KS22}). In particular, if $\K$ is connected, the Weyl group is a crystallographic reflection group, which justifies the name ``Weyl group''.
Moreover, the Chevalley Restriction Theorem (see, e.g., \cite[Thm~4.12]{PT87}) implies that restriction to the section $\mathfrak{a}$ of a polar $\K$-representation $V$ induces an isomorphism $\R[V]^{\K}\to \R[\mathfrak{a}]^W$ between the algebras of invariant polynomials. 

From the bijection $\mathfrak{a}/W\to V/\K$ between orbit spaces, it follows that intersection with $\mathfrak{a}$ induces a bijection between the set of $\K$-invariant subsets of $V$ and the set of $W$-invariant subsets of $\mathfrak{a}$. A key fact for us is that this bijection preserves convexity:

\begin{prop}\label{P:1-1}
    Let $V$ be a polar $\K$-representation, with section $\mathfrak{a}$ and Weyl group $W$. Then the maps in \eqref{eq:bijection-convex}, namely,
    \[\mathcal{O}\mapsto \mathcal{O}\cap \mathfrak{a}, \qquad S\mapsto \K\cdot S\]
    are inverse to each other, and establish a bijection between $\K$-invariant convex subsets $\mathcal{O}\subset V$ and $W$-invariant convex subsets $S\subset \mathfrak{a}$.
\end{prop}

\begin{proof}
This result follows from Kostant's Convexity Theorem (\cite[Thm.~8.2]{Kostant73}, see also \cite[Thm.~1, Cor.~2]{Mendes22}) when $\K$ is connected, and from \cite[Cor.~B]{Mendes22} in general, see also \cite[Thm~3.3 and Cor.~3.4]{KS22}.    
\end{proof}

\begin{ex}\label{ex:sym}
    Let $\K=\O(n)$ act by conjugation on the space $V=\Sym^2(\R^n)$ of symmetric $n\times n$ matrices endowed with the inner product $\langle A, B \rangle=\tr AB$; this is the isotropy representation of the symmetric space $\GL(n)/\O(n)$, with Cartan involution $A\mapsto (A^{-1})^T$. A section is given by the subspace $\mathfrak{a}$ of all diagonal matrices, which can be identified with $\R^n$. The corresponding Weyl group is the symmetric group $W=S_n$, which acts on $\mathfrak{a}$ by permuting the components. 
    
    Note that the orbits of $\K$ are all connected (provided $n\geq 2$), so $\K_0=\SO(n)$ and $\K$ have the same orbits, and, in particular, the corresponding Weyl groups coincide.

    In this example, the isomorphism $\R[V]^{\K}\to \R[\mathfrak{a}]^W$ yielded by the Chevalley Restriction Theorem can be seen as the familiar fact that the polynomials $\sum_i x_i^d$, for $d=1, \ldots , n$ generate the algebra of all symmetric polynomials in $x_1, \ldots, x_n$, and that their unique $\O(n)$-invariant extension to $V$ can be written as $A\mapsto \tr(A^n)$, and are hence polynomials.

    A convex subset $S$ of $\mathfrak{a}\cong\R^n$ that is invariant under $W=S_n$ is commonly called \emph{symmetric}. The corresponding $\O(n)$-invariant convex subset of $V$ is given by the set of symmetric matrices $A$ whose vector $\lambda(A)$ of eigenvalues (in, say, non-increasing order) belongs to $S$, and such sets are often called \emph{spectral convex sets}. In this context, Kostant's Convexity Theorem reduces to the Schur--Horn Theorem, which implies this special case of Proposition \ref{P:1-1}, see \cite[Lemma 2.1]{SS22}.
\end{ex}

\begin{rem}
    The bijection \eqref{eq:bijection-convex} induces a bijection between the faces of $\K$-in\-variant convex subsets of $V$ and the faces of $W$-invariant convex subsets of $\mathfrak a$, up to conjugation, and exposed faces are mapped to exposed faces, see~\cite{bgh1,bgh2}.
\end{rem}

\begin{rem}
    A bijection similar to \eqref{eq:bijection-convex} has been shown to hold in a more general context \cite[Cor.~B]{Mendes22}. Namely, given two Euclidean vector spaces $V,V'$, two submetries $V\to X$ and $V'\to X'$, and an isometry $X\to X'$, the latter induces a bijection between saturated closed convex subsets of $V$ and $V'$, respectively. In the polar case, one recovers Proposition \ref{P:1-1} by taking $V\to X$ and $V'\to X'$ to be the natural quotient maps $V\to V/\K$ and $\mathfrak{a}\to\mathfrak{a}/W$, and the isometry $V/\K\to \mathfrak{a}/W$ to be the natural bijection described after Definition \ref{D:polar}. %One may, of course, generalize the meta-question  to this situation \cite[Question 4]{Mendes22}.
\end{rem}

\section{Lifting convex algebro-geometric properties from a section}
\label{sec:proofthmb}

In view of Proposition \ref{P:1-1}, it is natural to ask which classes of convex subsets of Euclidean spaces are preserved under the bijection \eqref{eq:bijection-convex}. In this section, we answer this question affirmatively for convex semi-algebraic sets (Proposition~\ref{P:semialg}), spectrahedral shadows (Proposition~\ref{prop:shadows}), and rigidly convex sets (Proposition~\ref{prop:rigidlyconvex}), which proves Theorem~\ref{thm:main-polar} in the Introduction. Furthermore, we conjecture that spectrahedra are also preserved (Conjecture~\ref{con:specpolar}) and so are sums of squares (Conjecture~\ref{con:sos}).

\begin{prop}\label{P:semialg}
The bijection \eqref{eq:bijection-convex} maps convex semi-algebraic sets to convex semi-algebraic sets.
\end{prop}

\begin{proof}
    In view of Proposition \ref{P:1-1}, it suffices to show that the maps  \(\mathcal{O}\mapsto \mathcal{O}\cap \mathfrak{a}\) and \(S\mapsto \K\cdot S\) preserve semi-algebraicity. This is obvious for the first map, i.e., taking intersection with the vector subspace $\mathfrak{a}$, so we consider the second map. Let $S\subset\mathfrak{a}$ be a $W$-invariant semi-algebraic subset. This implies, e.g., by \cite{broecker}, that $S$ can be described by a boolean combination of inequalities involving only $W$-invariant polynomials. Replacing each of them by its preimage under the restriction map $\R[V]^{\K}\to \R[\mathfrak{a}]^W$, which exists and is unique by the Chevalley Restriction Theorem, gives a semi-algebraic description of a $\K$-invariant set $\cO$ that satisfies $\cO\cap\fa=S$. This implies that $\cO=\K\cdot S$, so the latter is semi-algebraic.
\end{proof}

\begin{prop}\label{prop:shadows}
    The bijection \eqref{eq:bijection-convex} maps spectrahedral shadows to spectrahedral shadows.
\end{prop}

\begin{proof}
Similarly to the previous case, this is clear for the map  \(\mathcal{O}\mapsto \mathcal{O}\cap \mathfrak{a}\), so we consider the map \(S\mapsto \K\cdot S\). First, we reduce to the case in which $\K$ is connected. Let $\K_1$ be the connected component of $\K$ containing the identity. The restriction of the representation $V$ to $\K_1$ is again polar with section $\mathfrak{a}$, and the corresponding Weyl group $W_1$ is a subgroup of $W$. Assume the conclusion is true for $\K_1$. 
    
Let $S\subset\mathfrak{a}$ be a $W$-invariant spectrahedral shadow. Then $S$ is also $W_1$-invariant, so that, by assumption, $\K_1\cdot S$ is a spectrahedral shadow. Let $g_1,\ldots, g_l\in \K$ be  representatives of the $\K_1$-cosets in $\K$. Then $\K\cdot S = g_1\K_1\cdot S \cup \cdots \cup g_l \K_1\cdot S$. Since $\K\cdot S$ is convex, it coincides with the convex hull of the union of the spectrahedral shadows $g_i\K_1\cdot S$, for $i=1,\ldots, l$. Therefore, $\K\cdot S$ is a spectrahedral shadow by \cite{NS09}, see also \cite[Thm.~2.2]{HN10}. We henceforth assume $\K$ is connected.

We use notation from \cite{KS22} in the rest of this proof. By Dadok's classification \cite[Prop.~1.2]{KS22}, we may assume $V=\mathfrak{p}$, where $\mathfrak{g}=\mathfrak{k}\oplus\mathfrak{p}$ is a Cartan decomposition of a real semi-simple Lie algebra $\mathfrak{g}$, $\mathfrak{k}$ is the Lie algebra of $\K$, $\mathfrak{a}$ is a maximal Abelian subspace of $\mathfrak{p}$, and the action of $\K$ on $V=\mathfrak{p}$ is given by the adjoint representation. Choose a Weyl chamber $C\subset \mathfrak{a}$. Let $S\subset\mathfrak{a}$ be a $W$-invariant spectrahedral shadow. Since $S$ is $W$-invariant and convex, 
    \[S=\bigcup_{x\in S\cap C} P_x,\]
    where $P_x$ is the \emph{moment polytope}, i.e., the convex hull of the orbit $W\cdot x$. Thus 
\begin{equation}
\K\cdot S=\bigcup_{x\in S\cap C} \mathcal{O}_x, \label{E:union}
\end{equation}
    where $\mathcal{O}_x=\K\cdot P_x$ is the \emph{polar orbitope}, i.e., the convex hull of the orbit $\K\cdot x$.

    From \cite[Thm.~4.2]{KS22}, polar orbitopes are spectrahedra. We need the more precise statement given in \cite[Thm.~4.4]{KS22}, namely,
\begin{equation} \label{E:orbitope}
y\in\mathcal{O}_x \iff \forall j=1, \ldots, n, \quad  c_j(x)\cdot \operatorname{Id}_{V_j}\succeq \rho^j(y),
\end{equation}
   where $n=\dim \mathfrak{a}$, the $\rho^j\colon \mathfrak{g}_\C\to \operatorname{End}(V_j)$ are certain irreducible (complex) representations of the complex Lie algebra $\mathfrak{g}_\C$, and $c_j(x)$ denotes the largest eigenvalue of $\rho^j(x)$. Note that for a certain choice of Hermitian inner product on $V_j$, the linear maps $\rho^j(z)$, for all $z\in V$, are self-adjoint (see \cite[Lemma~4.6(a)]{KS22}), so that the eigenvalues of $\rho^j(x)$ are all real, and hence ``$c_j(x)\cdot \operatorname{Id}_{V_j}\succeq \rho^j(y)$'' makes sense.

    By \cite[Lemma 4.6(b)]{KS22}, the function $C\to \R$ given by $x\mapsto c_j(x)$ is the restriction to $C$ of a linear function, say $L_j\colon\mathfrak{a}\to \R$. (In the notation of \cite{KS22}, $L_j$ is the restriction to $\mathfrak{a}$ of the highest weight $\omega_{i(j)}\in \mathfrak{h}^{\vee}$ of the representation $\rho^j$.) Define $N\subset V\times\mathfrak{a}$ by
    \[N=\{ (y,x) : L_j(x)\cdot \operatorname{Id}_{V_j}\succeq \rho^j(y) \quad \forall j=1,\cdots, n\}\]
    Since $\rho^j$ and $L_j$ are linear maps, $N$ is a spectrahedron. Let 
    \[N'=N\cap \big( V\times(S\cap C) \big).\]
    Since $S$ is a spectrahedral shadow by assumption, $C$ is a polytope, and $N$ is a spectrahedron, it follows that $N'$ is a spectrahedral shadow, because the class of spectrahedral shadows in closed under finite intersections. By \eqref{E:union} and \eqref{E:orbitope}, we have $\K\cdot S=\pi_1(N')$, where $\pi_1\colon V\times \mathfrak{a}\to V$ is the linear projection onto the first factor. Therefore $\K\cdot S$ is a spectrahedral shadow.
\end{proof}

\begin{rem}
    Proposition~\ref{prop:shadows} has been proved in the special case of Example \ref{ex:sym} in \cite[Thm.~4.1]{SS22}.
\end{rem}

We conjecture that the map \(S\mapsto \K\cdot S\), and hence the bijection \eqref{eq:bijection-convex}, also preserves the class of spectrahedra:

\begin{con}\label{con:specpolar}
    Let $\K$ be a compact Lie group and $V$ be a polar $\K$-representation with section $\mathfrak{a}\subset V$ and Weyl group $W$. If $S\subset \mathfrak{a}$ is a $W$-invariant spectrahedron, then $\K\cdot S\subset V$ is a spectrahedron. 
\end{con}

We note that Conjecture~\ref{con:specpolar} is still open even in the case of Example \ref{ex:sym}. One piece of evidence in this special case is \cite[Thm.~3.3]{SS22}, which states that if $S$ is a symmetric polytope, then the spectral convex set $\K\cdot S$ is a spectrahedron; further progress was made in \cite{Kummer21}. Moreover, recall that all faces of a spectrahedron are exposed, and the bijection \eqref{eq:bijection-convex} preserves the class of convex sets all of whose faces are exposed~\cite{bgh1,bgh2}.

Next, we show that the Generalized Lax Conjecture~\ref{conj:glc} implies Conjecture~\ref{con:specpolar}:

\begin{prop}\label{prop:rigidlyconvex}
The bijection \eqref{eq:bijection-convex} maps rigidly convex sets to rigidly convex sets.
\end{prop}

\begin{proof}
    As in previous cases, the property of being rigidly convex is clearly preserved by intersecting with a vector subspace, so we shall only consider the map \(S\mapsto \K\cdot S\). Let $S\subset \mathfrak a$ be a $W$-invariant rigidly convex set and $p$ be its defining polynomial as algebraic interior. Since $S$ is $W$-invariant and the defining polynomial is unique by Lemma~\ref{lem:defpoly}, $p$ is also $W$-invariant. By the Chevalley Restriction Theorem, there exists a unique $\K$-invariant polynomial $\widetilde p\in \R[V]^\K$ such that $\widetilde p|_{\mathfrak a}=p$. By Proposition~\ref{prop:rzprops} (i), we have that $p>0$ in the interior of $S$, and hence $\widetilde p>0$ in the interior of $\K\cdot S$. Moreover, $\widetilde p$ vanishes on $\partial (\K\cdot S)$ and has minimal degree among such polynomials because otherwise $p$ would not have minimal degree as the defining polynomial for $S$. Thus, $\K\cdot S$ is an algebraic interior with defining polynomial $\widetilde p$. It remains to show that $\widetilde p$ is real zero with respect to some (hence every) point in the interior of $\K\cdot S$. Without loss of generality, we may assume the origin $0\in \mathfrak a \subset V$ is in the interior of $S$, and hence $\widetilde p(0)=p(0)>0$. Given $v\in V$, let $g\in \K$ be such that $g\cdot v\in \mathfrak a$. Since $p$ is real zero with respect to $0\in\mathfrak a$, the univariate polynomial $t\mapsto p(t\,g\cdot v)$ has only real zeros. Since $\widetilde p(t\,v)=p(t\,g\cdot v)$, it follows that $t\mapsto \widetilde p(t\,v)$ has only real zeros. Thus, $\K\cdot S$ is rigidly convex.
\end{proof}

Closely related to spectrahedra and their shadows are polynomials that can be written as a sum of squares of polynomials \cite{SIAMbook}, which leads us to conjecture:

\begin{con}\label{con:sos}
 Let $p\in\R[\mathfrak a]^W$ be a $W$-invariant polynomial and $q\in\R[V]^\K$ the unique $\K$-invariant polynomial that restricts to $p$. Then $p$ is a sum of squares if and only if $q$ is a sum of squares.
\end{con}

Some partial evidence for this conjecture can be found in \cite[\S6]{Kummer21}.

\section{Optimization symmetry reduction with polar representations}\label{sec:optimization}

Symmetries of a convex optimization problem can be used to reduce the number of variables, eliminate degeneracies, and improve numerical conditioning; for a discussion in the context of
semidefinite programming (SDP), see \cite[\S 3.3.6]{SIAMbook}. 

\begin{symmred}\label{symred}
Consider an optimization problem whose feasible set $\mathcal O\subset V$ is a convex set invariant under a polar $\K$-representation on $V\cong\R^n$, say
\begin{equation}\label{eq:optimization-prob}
\begin{aligned}
    \text{maximize}\quad & v^T x, \\
    \text{subject to} \quad & x \in\mathcal O,
\end{aligned}
\end{equation}
for a given $v\in V$. Let $g\in\K$ be such that $g^{-1}\cdot v$ belongs to a given section $\mathfrak a\subset V$ of this polar representation; equivalently, $v\in g\cdot \mathfrak a$. 
Let $\Pi\colon V\to\mathfrak a$ be the orthogonal projection onto the section, and recall that $\Pi(\mathcal O)=\mathcal O\cap\mathfrak a$, by Kostant's Convexity Theorem.
Since $\mathcal O$ is $\K$-invariant, we have
\[\max_{x\in\mathcal O}\; v^T x = \max_{x\in\mathcal O}\; (g^{-1}\cdot v)^T x=   \max_{x\in\Pi(\mathcal O)} \;(g^{-1}\cdot v)^T x = \max_{x\in\mathcal O\cap \mathfrak a} \;(g^{-1}\cdot v)^T x,\]
that is, \eqref{eq:optimization-prob} is equivalent to the \emph{reduced} optimization problem
\begin{equation}\label{eq:optimization-prob-reduced}
\begin{aligned}
    \text{maximize}\quad & (g^{-1}\cdot v)^T x, \\
    \text{subject to} \quad & x \in\mathcal O\cap\mathfrak a,
\end{aligned}  
\end{equation}
whose feasible set $\mathcal O\cap \mathfrak a$ is a $W$-invariant convex subset of the subspace $\mathfrak a\subset V$. A point $x_0\in \mathcal O\cap \mathfrak a$ is an optimal solution to \eqref{eq:optimization-prob-reduced} if and only if all points in $\K\cdot x_0\subset V$ are optimal solutions to \eqref{eq:optimization-prob}.
\end{symmred}

We illustrate this by some well-known examples from Linear Algebra:

\begin{ex}\label{ex:lin-alg}
Let $\K=\SO(n)$, $n\geq2$, act by conjugation on $V=\Sym^2(\R^n)$, and recall this is a polar representation (Example~\ref{ex:sym}), with section $\mathfrak a\cong \R^n$ given by the subspace of diagonal matrices. Let $\mathcal O\subset V$ be an $\SO(n)$-invariant convex set, fix $A\in\Sym^2(\R^n)$ and choose $g\in\SO(n)$ such that $D=g^{-1}Ag$ is diagonal. Then,
\begin{equation*}
    \max_{X\in\mathcal O} \; \tr AX = \max_{X\in\mathcal O\cap \mathfrak a} \; \tr D X,
\end{equation*}
by Symmetry Reduction~\ref{symred}, where the optimization problem on the left-hand side has $\frac{n(n+1)}{2}$ variables, and the one on the right-hand side has $n$ variables. Note that this polar representation is not irreducible, namely $\Sym^2(\R^n)\cong \R\langle\Id_n\rangle\oplus\Sym^2_0(\R^n)$, but this is irrelevant for the above symmetry reduction.

A similar example is the adjoint representation of $\K=\SO(n)$ on its Lie algebra $V=\mathfrak{so}(n)\cong\wedge^2 \R^n$ with inner product $\langle A,B\rangle = \tr A B^T$, which is an irreducible polar representation; the isotropy representation of $\SO(n)$ as a symmetric space. In Linear Algebra terms, this is the $\SO(n)$-action by conjugation on skew-symmetric $n\times n$ matrices. 
The subspace $\mathfrak a\cong\R^m$ consisting of block diagonal matrices 
\begin{equation*}
\operatorname{diag}\left( \! 
\begin{pmatrix}
     0&\theta_1  \\
     -\theta_1& 0
\end{pmatrix},
\dots,
\begin{pmatrix}
     0&\theta_m  \\
     -\theta_m& 0
\end{pmatrix}\!\right)
\,\text{or}\,
\operatorname{diag}\left( \! 
\begin{pmatrix}
     0&\theta_1  \\
     -\theta_1& 0
\end{pmatrix},
\dots,
\begin{pmatrix}
     0&\theta_m  \\
     -\theta_m& 0
\end{pmatrix},0\!\right)
\end{equation*}
depending on the parity of $n$, where $m=\floor{n/2}$ and $\theta_j\in\R$ for all $j=1,\dots,m$, is a section.
Let $\mathcal O\subset V$ be an $\SO(n)$-invariant convex set, fix $A\in V$ and choose $g\in\SO(n)$ such that $D=g^{-1}Ag\in\mathfrak a$ is in the above normal form. Then,
\begin{equation*}
    \max_{X\in\mathcal O} \; \tr AX^T = \max_{X\in\mathcal O\cap \mathfrak a} \; \tr D X^T,
\end{equation*}
by Symmetry Reduction~\ref{symred}, where the optimization problem on the left-hand side has $\binom{n}{2}$ variables, and the one on the right-hand side has $\floor{n/2}$ variables.

The above examples essentially follow from the \emph{symmetric} and \emph{skew-symmetric Schur--Horn theorems}, cf.~\cite[\S 3]{orbitopes}. Kostant's Convexity Theorem gives a unified framework to handle all polar representations, e.g., the adjoint representations of matrix Lie groups, see \cite[\S 8]{KS22} for more examples with matrix Lie groups.
\end{ex}

One may hope that Symmetry Reduction~\ref{symred} would enlarge the class of convex optimization problems that can be solved with SDP, by replacing \eqref{eq:optimization-prob} with \eqref{eq:optimization-prob-reduced}. However, by Proposition~\ref{prop:shadows}, the convex set $\mathcal O\cap\mathfrak a$ is a spectrahedral shadow, i.e., \eqref{eq:optimization-prob-reduced} can be solved with SDP, if and only if $\mathcal O\subset V$ is a spectrahedral shadow, i.e., the original problem \eqref{eq:optimization-prob} can itself be solved with SDP. Nevertheless, 
some efficiency is gained, as the number of variables in the reduced problem \eqref{eq:optimization-prob-reduced} is $\dim \mathfrak a$, while the original problem has $\dim V$ variables. As discussed in Section~\ref{sec:prelimpolar}, polar representations $V$ of connected groups $\K$ are orbit-equivalent to the isotropy representation of a symmetric space. In this language, $\dim V$ is the dimension of the symmetric space, and $\dim \mathfrak a$ is its \emph{rank}. If $V$ is an irreducible $\K$-representation, it follows from the classification of symmetric spaces (see \cite[Chap.~X]{Hel78}) that $\dim \mathfrak a =O(\sqrt{\dim V})$. However, $\dim \mathfrak a$ may be larger if irreducibility is dropped; e.g., the natural action of $\K=\SO(2)^k$ on $V=\R^{2k}$ is polar with section $\mathfrak a =\R^k$.

\begin{rem}
Symmetry Reduction~\ref{symred} applies \emph{mutatis mutandis} to $\K$-representa\-tions $V$ of \emph{higher copolarity}, since the orthogonal projection $\Pi$ onto a generalized section $\mathfrak a$ also satisfies $\Pi(\mathcal O)=\mathcal O\cap \mathfrak a$, see \cite{Mendes22}.    
\end{rem}

\begin{ex}
Even if a $\K$-representation $V$ has a vector subspace $\mathfrak a\subset V$ that meets all $\K$-orbits, Symmetry Reduction~\ref{symred} may fail if the assumption that $V$ is \emph{polar} is dropped. Namely, consider the Hopf action of $\mathsf U(1)$ on $V=\C^2$ given by $e^{i t}\cdot (z_1,z_2)=(e^{it}z_1,e^{it}z_2)$, and note that the vector subspace $\mathfrak a\cong\R^3$ consisting of $(z_1,z_2)\in\C^2$ with $\operatorname{Im} z_2=0$ meets all $\mathsf U(1)$-orbits (but not orthogonally). Although
 \begin{equation}\label{eq:fail}
     \max_{x\in\mathcal O}\; v^T x = \max_{x\in\mathcal O}\; (g^{-1}\cdot v)^T x=  \max_{x\in\Pi(\mathcal O)} \;(g^{-1}\cdot v)^T x,
 \end{equation} 
 where $\Pi\colon\C^2\to\mathfrak a$ is the orthogonal projection, in general $\Pi(\mathcal O)\supsetneq \mathcal O\cap\mathfrak a$. For example, the convex hull $\mathcal O$ of the $\mathsf U(1)$-orbit of $(1,1)\in\C^2$ is such that $\mathcal O\cap\mathfrak a$ is the straight line segment with endpoints $\pm (1,1)\in\C^2$, while $\Pi(\mathcal O)$ is the convex hull of the ellipse $\{(e^{it},\cos t)\in\C^2:t\in\R\}$. Note that the complexity of finding $\Pi(\mathcal O)$ is the same as the original optimization problem, so nothing is gained by using \eqref{eq:fail}.
\end{ex}

\begin{rem}
To implement Symmetry Reduction~\ref{symred}, one must find $g\in\K$ such that $g^{-1}\cdot v$ belongs to $\mathfrak a\subset V$, or, alternatively, find a section $\mathfrak a'\subset V$ with $v\in \mathfrak a'$. For generic $v\in V$, the latter can be achieved by setting $\mathfrak a'$ to be the orthogonal complement of the tangent space at $v$ to the $\K$-orbit through $v$. In practice, this amounts to performing matrix multiplications and solving a linear systems of equations. However, in some situations (such as Example~\ref{ex:lin-alg}), the former problem reduces to diagonalizing a symmetric matrix, for which faster methods are available.
\end{rem}

\def\cprime{$'$}
\providecommand{\bysame}{\leavevmode\hbox to3em{\hrulefill}\thinspace}
\providecommand{\MR}{\relax\ifhmode\unskip\space\fi MR }
\providecommand{\MRhref}[2]{%
  \href{http://www.ams.org/mathscinet-getitem?mr=#1}{#2}
}
\providecommand{\href}[2]{#2}


\begin{thebibliography}{BKM21}

\bibitem[Bar02]{barvinok}
{\sc A.~Barvinok}, \emph{A course in convexity}, Graduate Studies in
  Mathematics, vol.~54, American Mathematical Society, Providence, RI, 2002.
  \MR{1940576}

\bibitem[BKM21]{BKM-Siaga}
{\sc R.~G. Bettiol, M.~Kummer, and R.~A.~E. Mendes}, \emph{Convex algebraic
  geometry of curvature operators}, SIAM J. Appl. Algebra Geom. \textbf{5}
  (2021), no.~2, 200--228. \MR{4252070}


\bibitem[BGH13]{bgh1}
{\sc L.~Biliotti, A.~Ghigi, and P.~Heinzner}, \emph{Polar orbitopes}, Comm.
  Anal. Geom. \textbf{21} (2013), no.~3, 579--606. \MR{3078948}

\bibitem[BGH16]{bgh2}
{\sc \bysame}, \emph{Invariant convex sets in polar representations}, Israel J.
  Math. \textbf{213} (2016), no.~1, 423--441. \MR{3509478}


\bibitem[BPT13]{SIAMbook}
{\sc G.~Blekherman, P.~A. Parrilo, and R.~R. Thomas (eds.)}, \emph{Semidefinite
  optimization and convex algebraic geometry}, MOS-SIAM Ser.\ Optim.\ 13, SIAM, Philadelphia, 2013. \MR{3075433}

\bibitem[Br{\"o}98]{broecker}
{\sc L.~Br{\"o}cker}, \emph{On symmetric semialgebraic sets and orbit spaces},
  Singularities symposium -- {\L}ojasiewicz 70.
% Papers from the Symposium on Singularities held at Jagellonian University, Krak\'ow, September 25–29, 1996 and the Seminar on Singularities and Geometry held in Warsaw, September 30–October 4, 1996. Edited by Bronis{\l}aw Jakubczyk, Wies{\l}aw Paw{\l}ucki and Jacek Stasica. 
  Banach Center Publ., 44. Polish Academy of Sciences, Institute of Mathematics, Warsaw, 1998,
  pp.~37--50.

\bibitem[BtD95]{BtD}
{\sc T.~Br\"{o}cker and T.~tom Dieck}, \emph{Representations of compact {L}ie
  groups}, Graduate Texts in Mathematics, vol.~98, Springer-Verlag, New York,
  1995, Translated from the German manuscript, Corrected reprint of the 1985
  translation. \MR{1410059}

\bibitem[Dad85]{Dadok85}
{\sc J.~Dadok}, \emph{Polar coordinates induced by actions of compact {L}ie
  groups}, Trans. Amer. Math. Soc. \textbf{288} (1985), no.~1, 125--137.
  \MR{773051}

\bibitem[GP04]{gatermann-parrilo}
{\sc K.~Gatermann and P.~A. Parrilo}, \emph{Symmetry groups, semidefinite
  programs, and sums of squares}, J. Pure Appl. Algebra \textbf{192} (2004),
  no.~1-3, 95--128. \MR{2067190}

\bibitem[Hel78]{Hel78}
 {\sc S.~Helgason}, \emph{Differential geometry, Lie groups, and symmetric spaces}, Pure and Applied Mathematics, 80, Academic Press, New York-London, 1978 \MR{0514561}

\bibitem[HN10]{HN10}
{\sc J.~W. Helton and J.~Nie}, \emph{Semidefinite representation of convex
  sets}, Math. Program. \textbf{122} (2010), no.~1, 21--64. \MR{2533752}

\bibitem[HV07]{HV07}
{\sc J.~W. Helton and V.~Vinnikov}, \emph{Linear matrix inequality
  representation of sets}, Comm. Pure Appl. Math. \textbf{60} (2007), no.~5,
  654--674. \MR{2292953}

\bibitem[KS22]{KS22}
{\sc T.~Kobert and C.~Scheiderer}, \emph{Spectrahedral representation of polar
  orbitopes}, Manuscripta Math. \textbf{169} (2022), no.~1-2, 185--208.
  \MR{4462662}


\bibitem[Kos73]{Kostant73}
{\sc B.~Kostant}, \emph{On convexity, the {W}eyl group and the {I}wasawa
  decomposition}, Ann. Sci. \'{E}cole Norm. Sup. (4) \textbf{6} (1973),
  413--455. \MR{364552}


\bibitem[Kum16]{2results}
{\sc M.~Kummer}, \emph{Two results on the size of spectrahedral descriptions},
  SIAM J. Optim. \textbf{26} (2016), no.~1, 589--601. \MR{3463702}

\bibitem[Kum21]{Kummer21}
{\sc \bysame}, \emph{Spectral linear matrix inequalities}, Adv. Math.
  \textbf{384} (2021), Paper No. 107749, 36. \MR{4246101}

\bibitem[{Men}22]{Mendes22}
{\sc R.~A.~E. {Mendes}}, \emph{{A geometric take on Kostant's Convexity
  Theorem}}, Transform. Groups, to appear. arXiv:2202.12966.

\bibitem[NP23]{netzer-plaumann}
{\sc T.~Netzer and D.~Plaumann}, \emph{Geometry of linear matrix
  inequalities---a course in convexity and real algebraic geometry with a view
  towards optimization}, Compact Textbooks in Mathematics,
  Birkh\"{a}user/Springer, Cham, 2023. \MR{4633251}

\bibitem[NS09]{NS09}
{\sc T.~Netzer and R.~Sinn}, \emph{A note on the convex hull of finitely many
  projections of spectrahedra}, arXiv (2009), arXiv:0908.3386.

\bibitem[PT87]{PT87}
{\sc R.~S. Palais and C.-L. Terng}, \emph{A general theory of canonical forms},
  Trans. Amer. Math. Soc. \textbf{300} (1987), no.~2, 771--789. \MR{876478}

\bibitem[SS25]{SS22}
{\sc R.~{Sanyal} and J.~{Saunderson}}, \emph{{Spectral Polyhedra}}, Forum Math.\ Sigma 13 (2025), Paper No.\ e30.

\bibitem[SSS11]{orbitopes}
{\sc R.~Sanyal, F.~Sottile, and B.~Sturmfels}, \emph{Orbitopes}, Mathematika
  \textbf{57} (2011), no.~2, 275--314. \MR{2825238}

\bibitem[Sch18]{shadows}
{\sc C.~Scheiderer}, \emph{Spectrahedral shadows}, SIAM J. Appl. Algebra Geom.
  \textbf{2} (2018), no.~1, 26--44. \MR{3755652}


\bibitem[Val09]{vallentin}
{\sc F.~Vallentin}, \emph{Symmetry in semidefinite programs}, Linear Algebra
  Appl. \textbf{430} (2009), no.~1, 360--369. \MR{2460523}

\bibitem[Vin12]{pastpresentfuture}
{\sc V.~Vinnikov}, \emph{{LMI} representations of convex semialgebraic sets and
  determinantal representations of algebraic hypersurfaces: past, present, and
  future}, Mathematical methods in systems, optimization, and control.
  Festschrift in honor of J. William Helton, Birkh{\"a}user, Basel, 2012,
  pp.~325--349.

\end{thebibliography}
 \end{document}